\documentclass[twoside,reqno]{amsart}

\usepackage{amsfonts,amsmath,amscd,amsthm,amssymb}
\usepackage{mathtools}
\usepackage[mathscr]{euscript} 
\usepackage{graphics}
\usepackage{wrapfig}
\usepackage{lscape}
\usepackage{rotating}
\usepackage{epsfig}
\usepackage{cite}
\usepackage{color}
\usepackage{booktabs}
\usepackage[vlined,ruled]{algorithm2e}
\usepackage{multirow}
\usepackage{lipsum,multicol}

\usepackage{tikz}
\usetikzlibrary{patterns}

\newtheorem{lemma}{Lemma}
\newtheorem{theorem}{Theorem}
\newtheorem{corollary}{Corollary}
\newtheorem{remark}{Remark}

\def\restrict#1{\raise-0.2ex\hbox{\ensuremath|}_{#1}}

\newcommand{\bld}[1]{\boldsymbol{#1}}

\newcommand{\compPhi}{\underline{{\phi}}}
\newcommand{\compPsi}{\underline{\psi}}

\newcommand{\compXi}{\underline{{\bld \xi}}}
\newcommand{\compEta}{\underline{\bld \eta}}

\newcommand{\compP}{\underline{{p}}}

\newcommand{\compPspace}{\underline{W}(h)}
\newcommand{\compPhspace}{\underline{W\!}_h}
\newcommand{\compPh}{\underline{{p}}_h}

\newcommand{\compU}{\underline{\boldsymbol{u}}}
\newcommand{\compUh}{\underline{\boldsymbol{u}}_h}
\newcommand{\compUP}{\underline{\Pi \boldsymbol{u}}}
\newcommand{\compPP}{\underline{\Pi {p}}}

\newcommand{\compUhspace}{\underline{\boldsymbol{U}\!}_h}
\newcommand{\compUspace}{\underline{\boldsymbol{U}}(h)}

\newcommand{\compV}{\underline{\boldsymbol{v}}}
\newcommand{\compWW}{\underline{\boldsymbol{w}}}
\newcommand{\compW}{\underline{{w}}}

\newcommand{\compVh}{\underline{\boldsymbol{v}}_h}
\newcommand{\eu}{{\bld \varepsilon}_{u}}
\newcommand{\du}{{\bld\delta}_{u}}
\newcommand{\compep}{\underline{\varepsilon}_{p}}

\newcommand{\dpp}{{ \delta}_{p}}
\newcommand{\ep}{{ \varepsilon}_{p}}
\newcommand{\compeu}{\underline{\bld \varepsilon}_{u}}

\newcommand{\Vh}{\bld{V}_{\!h}}
\newcommand{\Wh}{{W}_{\!h}}
\newcommand{\wf}{\widehat{w}}
\newcommand{\wwf}{\widehat{\bld w}}

\newcommand{\vf}{\widehat{\bld v}}

\newcommand{\pf}{\widehat{p}}

\newcommand{\uf}{\widehat{\bld u}}

\newcommand{\psif}{\widehat{\psi}}
\newcommand{\phif}{\widehat{\phi}}
\newcommand{\xif}{\widehat{\bld \xi}}
\newcommand{\etaf}{\widehat{\bld \eta}}

\newcommand{\WFh}{\widehat{W}_{\!h}}
\newcommand{\VFh}{\widehat{\bld V}_{\!h}}

\newcommand{\PiWF}{\Pi_{\widehat{W}}}
\newcommand{\PiVF}{\Pi_{\widehat{V}}}

\newcommand{\divs}{{\mathrm{div}}}
\newcommand{\grads}{{\nabla}}

\newcommand{\pol}{\mathbb{P}}

\newcommand{\bR}{\mathbb{R}}
\newcommand{\Oh}{\mathcal{T}_h}

\newcommand{\Eh}{\mathcal{F}_h}

\newcommand{\jmp}[1]{[\![#1 ]\!]}
\newcommand{\vertiii}[1]{{\left\vert\kern-0.25ex\left\vert\kern-0.25ex\left\vert #1
    \right\vert\kern-0.25ex\right\vert\kern-0.25ex\right\vert}}

\newcommand{\gradss}{\nabla_s}
\begin{document}
\title[Divergence-conforming HDG for The Biot Problem]{
A high-order HDG method for the Biot's consolidation model}
\author{Guosheng Fu}
\address{Division of Applied Mathematics, Brown University, 182 George St,
Providence RI 02912, USA.}
\email{Guosheng\_Fu@brown.edu}



\begin{abstract}
We propose a novel high-order HDG method for the Biot's consolidation model in poroelasticity.
We present optimal error analysis for both the semi-discrete and full-discrete (combined with temporal backward differentiation formula) schemes.
Numerical tests are provided to demonstrate the performance of the method.
\end{abstract}
\maketitle

\section{Introduction}
\label{sec:intro}
Biot's seminar work \cite{Biot41,Biot56,Biot72} laid the foundation of the theory of poroelasticity, which models the 
the interaction between the fluid flow and deformation in an fluid-saturated porous medium.
The model is used in several industries such as petroleum and environmental engineering \cite{Zheng03,Strehlow15} 
and medical applications such as the modeling of the intestinal oedema \cite{Young14}.

In this paper, we consider the numerical solution of the 
following quasi-static Biot's consolidation model 
\begin{subequations}
\label{eqns}
\begin{alignat}{3} 
\label{eq1}
c_s \dot{p}+\alpha\, \divs(\dot{\bld u}) - \divs(\kappa\grads p) &= f \qquad && \text{in $\Omega$,}\\
\label{eq2}
- \mathrm{div}\left(2\mu\gradss(\bld u) - \lambda\,\mathrm{div}(\bld u)\bld I\right)  
+\alpha\grads p&= 
\bld g \qquad && \text{in $\Omega$,}
\end{alignat}
with homogeneous Dirichlet boundary conditions and proper initial data:
\begin{alignat}{3}
\label{eq3}
\bld u &= \bld 0, && \quad\quad\quad p \;= 0\quad\quad\quad&&\text{on $\partial \Omega$,}\\
\label{eq4}
\bld u(0,\bld x) &= \bld u_0(\bld x) , &&\quad\quad\quad p(0,\bld x) \;= p_0(\bld x)\quad\quad\quad
&& \text{in $ \Omega$,}
\end{alignat}
\end{subequations}
where $\Omega\subset \bR^d$, $d=2,3$, is a bounded polygonal/polyhedral domain,
$p$ is the pressure and $\bld u$ is the deformation,  
$c_s\ge 0$ is the constrained specific storage coefficient which is close to {\it zero}
in many applications, 
 $\alpha$ is the Biot-Willis constant which is close to {\it one},
 $\kappa$ is the permeability tensor,
$\lambda$ and $\mu$ are the Lam\'e constants, 
and $\gradss\bld u = {(\grads\bld u + 
\grads^T\bld u)}/{2}$ is the symmetric gradient operator.
Here we consider homogeneous boundary condition for simplicity. More general boundary 
conditions, c.f. \cite{Phillips07a}, can be handled with minor modification.

There are extensive literature on the study of spatial discretization for the  Biot's consolidation model with the finite element methods.
The early work of Murad et. al. \cite{MuradLoula92, MuradLoula94,Murad96} studied the stability of the scheme using 
stable pair of Stokes finite elements for displacement and pressure.
The monograph of Lewis and Schrefler \cite{Lewis98}, c.f. also references therein,
discussed the finite element discretization using continuous Galerkin method for both the the displacement and pressure.
Phillips and Wheeler proposed and analyzed an algorithm that combines 
the mixed methods for pressure and a continue/discontinuous Galerkin method for displacement 
\cite{Phillips07a, Phillips07b,Phillips08}. 
See also the discontinuous Galerkin methods \cite{Liu09, Chen13,Wheeler14a, Riviere17},
Galerkin least square method \cite{Korsawe05},  
the pressure-stabilized methods \cite{Wan02,White08,Berger15,Berger17}, 
the mixed methods \cite{Ferronato10, Yi14, Lee16,Oyarzua16, Lee17,Lee17b}, and the nonconforming methods \cite{Yi13, Boffi16, Hu17}.

In this paper, we consider the discretization to \eqref{eqns} using a displacement-pressure formulation with a high-order, superconvergent HDG method 
for the pressure Poisson operator \cite{Lehrenfeld:10, Oikawa:15}, and a high-order, divergence-conforming 
HDG method for the elasticity operator \cite{Lehrenfeld:10,LehrenfeldSchoberl16, FuLehrenfeld18, FuLehrenfeld18a}.
The resulting differential algebraic system (DAE) is solved using  backward differentiation formula (BDF) \cite{HairerWanner10}.
We present optimal a priori error estimates for the resulting semi-discrete and full-discrete schemes. The method is proven to be 
free from 
Poisson locking as $\lambda\rightarrow \infty$, and is numerically shown to be also free from pressure oscillation in the case of
low permeability with small time step size \cite{Phillips09}.
To reach a convergence rate of $k+1$ for an energy norm, the fully discrete scheme has a set of globally coupled degrees of freedom (after static
condensation) consists of 
polynomials of degree $k+1$ for the normal displacement, 
polynomials of degree $k$ for the tangential displacement, and 
polynomials of degree $k-1$ for the pressure per facet (edge in 2D, face in 3D). 
We also discuss an improvement of this base scheme by slightly  relaxing the $H(\mathrm{div})$-conformity of the displacement space so that 
only unknowns of polynomial degree $k$ are involved for normal-continuity, c.f. \cite{Lederer17, FuLehrenfeld18a}. 
This modification results a globally coupled degrees of freedom  
consists of (vector) polynomials of degree $k$ for the displacement, and degree $k-1$ for the pressure per facet.
It does not deteriorate the convergence rate,  and allow for optimality of the method also in the sense of 
superconvergent HDG methods. 

The rest of the paper is organized as follows. In Section \ref{sec2:disc}, the semi-discrete scheme is introduced and analyzed. 
In Section \ref{sec3:time}, the fully-discrete scheme is introduced and analyzed. The numerical results supporting the theory is 
presented in Section \ref{sec:numerics}. And a conclusion is drawn in Section \ref{sec:conclusion}.

\section{Semi-discrete Scheme}
\label{sec2:disc}
\subsection{Preliminaries}
Let $\Oh=\{T\}$ be a conforming simplicial triangulation of $\Omega$.
Let $\Eh=\{F\}$ be the collection of facets (edges in 2D, faces in 3D) in $\Oh$.
For any element $T \in\Oh$, we denote by $h_T$ its diameter and we denote by $h$ the
maximum diameter over all mesh elements. 

We distinguish functions with support only on facets indicated by a {\it hat } notation, 
e.g. $\widehat {\phi}$, $\widehat{\bld \xi}$,  with functions with support also on the volume elements.
Compositions of functions supported on volume elements (without {\it hat } notation)
and those only on facets 
are used for the HDG discretization and indicated by underlining, e.g.,
$\compPhi = (\phi ,\widehat{\phi})$,
$\compXi = (\bld \xi ,\widehat{\bld \xi})$.
To simplify notation, we denote the compound spaces 
\begin{align*}
  \compPspace : = &\;H_0^2(\Omega)\times H_0^1(\Eh), \text{ and }\;
 \compUspace : = [H_0^2(\Omega)]^d\times [H_0^1(\Eh)]^d.
\end{align*}
We denote the tangential component of a vector 
$\bld v$ on a facet $F$ by $(\bld v)^t = \bld v-(\bld v\cdot \bld n)\bld n$, where $\bld n$ is the normal direction on $F$.
Furthermore, for any function $\phi\in H^2_0(\Omega)$, we denote $\compPhi :=(\phi, \phi|_{\Eh})\in \compPspace$, 
and for any function  $\bld \xi\in [H^2_0(\Omega)]^d$, we denote 
$\compXi :=(\bld \xi, (\bld \xi)^t|_{\Eh}) \in \compUspace$.

For a domain $D\in \bR^d$, 
we denote $(\cdot,\cdot)_D$ as the standard $L^2$-inner product on $D$. Whenever there is no confusion, 
we simply denote $(\cdot,\cdot)$ as the inner product on the whole domain $\Omega$.

Finally, to simplify the presentation of our analysis, we assume the permeability tensor $\kappa$ is a constant 
scalar throughout the domain $\Omega$. 
However, we note that the method is applicable to the more general case of a fully tensorial (possibly piecewise defined) permeability.

\subsection{Finite elements}
We 
consider
an HDG method which approximates 
the pressure and displacement on the mesh 
$\Oh$,
and the pressure and {\it tangential } component of the displacement on 
the mesh skeleton $\Eh$:

\begin{subequations}
\label{space}
\begin{align}
\label{space-1}
\Wh : =&\; \prod_{T\in\Oh}\pol^{k}(T),\\
\WFh := &\;\{\wf \in \prod_{F\in\Eh}\pol^{k-1}(F), \;\;
\wf = 0 \,\;\;\forall F\subset \partial\Omega\},
\\
\Vh : =&\; \{\bld v\in \prod_{T\in\Oh}[\pol^{k+1}(T)]^d, \;\;
\jmp{\bld v\cdot\bld n}_F = 0 \,\forall F\in\Eh\}\subset H_0(\mathrm{div},\Omega),\\
\VFh := &\;\{\vf\in \prod_{F\in\Eh} [\pol^{k}(F)]^d, \;\;
\vf\cdot\bld n = 0 \,\forall F\in\Eh, \;\;
\vf = 0 \,\;\;\forall F\subset \partial\Omega\},
\end{align}
where $\jmp{\cdot}$ is the usual jump operator, 
$\pol^m$ the space of polynomials up to degree $m$.
Note that functions in $\WFh$ and $\VFh$ are defined only on the mesh skeleton, and the
normal component of functions in $\VFh$ is {\it zero}.
Here the polynomial degree $k\ge 1$ is a positive integer. 
%

%

To further simplify notation, we denote the composite spaces 
\begin{alignat*}{2}
 \compPhspace : =&\; \Wh\times \WFh, &&\text{ and }
 \compUhspace : = \Vh\times \VFh.
\end{alignat*}

\end{subequations}

\subsection{The semi-discrete numerical scheme}
First, we introduce the following $L^2$ projections on the facets:
 \begin{align*}
    \PiWF: L^2(F)\rightarrow \pol_{k-1}(F), 
  \quad \int_F (\PiWF f) w \, \mathrm{ds} = \int_{F}f\,w\, \mathrm{ds} \quad \forall w\in \pol_{k-1}(F),\\
    \PiVF: [L^2(F)]^d\rightarrow [\pol_k(F)]^d, 
  \quad \int_F (\PiVF \bld f)\bld v \, \mathrm{ds} = \int_{F}\bld f\,\bld v\, \mathrm{ds} \quad \forall \bld v\in [\pol_k(F)]^d.
 \end{align*}
 Then, for all $\compPhi=(\psi,\psif), \compPsi=(\psi,\psif)\in \compPhspace+\compPspace$, and 
 $\compXi=(\bld \xi,\xif), \compEta=(\bld \eta, \etaf)\in \compUhspace+\compUspace$, 
 we introduce the bilinear forms for the diffusion and elasticity operators, respectively, 
\begin{subequations}
\label{bilinearforms}
 \begin{align}
  \label{diffusion-1}
  a_{h}(\compPhi, \compPsi) :=&\; 
  \sum_{T\in\Oh}\int_T\kappa\,\grads \phi\cdot\grads \psi
  -\int_{\partial T}\kappa\,\grads \phi\cdot\bld n\, \jmp{\compPsi}\,\mathrm{ds}\\
&\; 
-\int_{\partial T}\kappa\,\grads \psi\cdot \bld n\jmp{\compPhi}\,\mathrm{ds}+ 
\int_{\partial T}\kappa\frac{\tau}{h}\PiWF\jmp{\compPhi}\, \PiWF\jmp{\compPsi}\,\mathrm{ds},
\nonumber\\
  \label{elas-1}
  b_{h}(\compXi, \compEta) :=&\; 
  \sum_{T\in\Oh}\int_T2\mu\,\gradss(\bld \xi):\gradss(\bld \eta)
+  \lambda\,\mathrm{div}(\bld \xi)\mathrm{div}(\bld \eta)
  \,\mathrm{dx}\\
  &\;  -\int_{\partial T}2\mu\,\gradss(\bld \xi)\bld n\cdot \jmp{\compEta^t}\,\mathrm{ds} 
-\int_{\partial T}2\mu\,\gradss(\bld \eta)\bld n\cdot \jmp{\compXi^t}\,\mathrm{ds}\nonumber\\
&\;  + \int_{\partial T}\mu\frac{\tau}{h}\PiVF\jmp{\compXi^t}\cdot \PiVF\jmp{\compEta^t}\,\mathrm{ds},
\nonumber
 \end{align}
\end{subequations}
 where 
  $\jmp{\compPhi}= \phi-\phif$ and 
 $\jmp{\compXi^t}= (\bld \xi)^t-\xif$ denote 
 the jumps between interior and facet unknowns, and 
 $\tau = \tau_0 k^2$ with $\tau_0$ a sufficiently large positive constant. 
 
 We note that as long as $\phi$ and 
 $\bld \xi$ are finite element functions in 
 $\Wh$ and $\Vh$, respectively, we have
 \begin{align}
 \label{idd1}
  \int_{\partial T} \kappa \grads  \phi \cdot\bld n \, \jmp{\compPsi} ds
= &\;\int_{\partial T} \kappa \grads\phi \cdot\bld n \, \PiWF \jmp{\compPsi} ds,\\
 \label{idd2}
  \int_{\partial T} 2 \mu \gradss(\bld \xi) \bld n \cdot \jmp{\compEta^t} ds
=&\; \int_{\partial T} 2 \mu \gradss(\bld \xi) \bld n \cdot \PiVF \jmp{\compEta^t} ds
 \end{align}
as 
$\kappa \grads \phi\cdot \bld n$ 
is a polynomial of degree $k-1$, and 
$2 \mu \gradss(\bld \xi) \bld n$ 
is a polynomial of degree $k$ on each facet.
 
The semi-discrete numerical scheme then reads: 
Find $\compPh = (p_h,\pf_h)\in \compPhspace$ and 
$\compUh=(\bld u_h,\uf_h)\in \compUhspace$ such that 
\begin{subequations}
\label{scheme}
\begin{align}
\label{scheme-1}
(c_s \dot{p}_{h}+\alpha\,\divs(\dot{\bld u}_h), w_h)
+
a_h(\compPh, \compW_h) =&\; (f, w_h), \quad \forall \compW_h=(w_h,\wf_h) \in \compPhspace,\\
\label{scheme-2}
b_h(\compUh, \compVh)
-(p_h,\alpha\,\divs(\bld v_h))=&\; (\bld g, \bld v_h), \quad \forall \compVh=(\bld v_h,\vf_h) \in \compUhspace.
\end{align}
\end{subequations}

\subsection{Semi-discrete error estimates}
We write
\[
 A\preceq B
\]
to indicate that there exists a constant $C$, independent of the mesh size $h$, the parameters
$c_s,\alpha, \mu,\lambda,\kappa$
and the numerical solution, such that 
$A\le CB$.

We denote the following (semi)norms:
\begin{subequations}
\label{norms}
\begin{align}
 \label{norm-p}
 \|\compW\|_{1,h} := &\;
 \left(
\sum_{T\in\Oh} \|\grads w\|^2_T
+\frac{1}{h}\|\jmp{\compW}\|^2_{\partial T}
 \right)^{1/2},\\
 \label{norm-u}
 \|\compV\|_{\mu,h} := &\;
 \left(
\sum_{T\in\Oh} 2\mu\|\gradss \bld v\|^2_T
+\frac{2\mu}{h}\|\PiVF\jmp{\compV^t}\|^2_{\partial T}
 \right)^{1/2},\\
 \label{norm-energy2}
 \|\compV\|_{\mu,*,h} := &\;
\Big(
\|\compV\|_{\mu,h}^2+
\sum_{T\in\Oh} 2\mu h
\|\gradss(\bld v)\bld n\|^2_{\partial T}
 \Big)^{1/2},\\
  \label{norm-total}
 \vertiii{\{\compW, \compV\}}_{h} := &\;
 \left(c_s\|w\|^2+\|\compV\|_{\mu,h}^2
 +\lambda \|\divs\,\bld u\|^2 \right)^{1/2}.
\end{align}
\end{subequations}
We also denote the $H^{s}$-norm  on $\Omega$ as $\|\cdot\|_{s}$, and when 
$s=0$, we simply denote $\|\cdot\|$ as the $L^2$-norm on $\Omega$.

Coercivity of the bilinear forms \eqref{bilinearforms} follows directly from 
\cite{Oikawa:15,FuLehrenfeld18a}.
\begin{lemma}
 \label{lemma:coercivity}
Let the stabilization parameter $\tau_0$ be sufficiently large. 
Then, for any function $\compW_h\in \compPhspace$, there holds
 \begin{subequations}
 \label{coercivity}
 \begin{align}
 \label{coercivity-1}
\kappa\|\compW_h\|_{1,h}^2  \preceq a_h(\compW_h,\compW_h),
 \end{align}
 and for any function $\compVh\in\compUhspace$, there holds
 \begin{align}
 \label{coercivity-2}
\|\compVh\|_{\mu,h}^2
+\lambda\|\divs\, \bld v_h\|^2
\preceq b_h(\compVh,\compVh).
 \end{align} 
 \end{subequations}

\end{lemma}

Consistency of the semi-discrete scheme \eqref{scheme} follows directly from integration by parts.
\begin{lemma}
\label{lemma:consistency}
Let $(p, \bld u)\in H_0^2(\Omega)\times \bld H_0^2(\Omega)$ be the solution to the equations \eqref{eqns}.
We have 
\begin{align*}
(c_s \dot{p}+\alpha\,\divs(\dot{\bld u}), w)
+
a_h(\compP, \compW) =&\; (f, w), \quad \forall \compW_h=(w,\wf) \in \compPhspace+\compPspace,\\
b_h(\compU, \compV)
-(p,\alpha\,\divs(\bld v))=&\; (\bld g, \bld v), \quad \forall \compV=(\bld v,\vf) \in \compUhspace
+\compUspace.
\end{align*}
\end{lemma}

We use the technique of {\it elliptic  projectors} \cite{Wheeler75} to derive optimal convergent error estimates. 
Let $\compPP = (\Pi p, \widehat{\Pi}p)\in \compPhspace$ and 
$\compUP = (\Pi \bld u, \widehat{\Pi}\bld u)\in \compUhspace$ be the 
projectors defined as follows:
\begin{subequations}
\label{ellip-proj}
\begin{alignat}{3}
\label{ellip-proj-1}
a_h(\compPP-\compP, \compW_h) =&\; 0, &&\quad \forall \compW_h=(w_h,\wf_h) \in \compPhspace,\\
\label{ellip-proj-2}
b_h(\compUP-\compU, \compVh)
-(\Pi p-p,\alpha\,\divs(\bld v_h))=&\; 
0, &&\quad \forall \compVh=(\bld v_h,\vf_h) \in \compUhspace. 
\end{alignat}
\end{subequations}
Note that the above coupling is weak since 
the pressure projector is purely determined by the first set of equations \eqref{ellip-proj-1}.

The approximation properties of these elliptic projectors follows directly from the 
corresponding analysis for the elliptic problems \cite{Oikawa:15, FuLehrenfeld18a}.

We shall assuming the following full $H^2$-regularity
 \begin{align}
 \label{dual}
\|\phi\|_{2}
\preceq \|\theta\|
 \end{align}
  for the dual problem
$ -  \grads(\kappa\,\grads \phi )  =\; \theta$ with homogeneous Dirichlet boundary conditions
 for any source term $\theta\in L^2(\Omega)$.
The estimate \eqref{dual} holds on convex domains.

\begin{lemma}
 \label{lemma:approx}
Let the stabilization parameter $\tau_0$ be sufficiently large. 
Let $\compPP\in \compPhspace$ and $\compUP\in \compUhspace$ be given by \eqref{ellip-proj}. 
 Assume the elliptic regularity result \eqref{dual} holds. Then, 
the following estimates holds:
 \begin{subequations}
  \label{approx:est}
  \begin{align}
  \label{approx:est-1}
   \|p-\Pi p\|\preceq&\; h^{k+1}\|p\|_{k+1},\\
  \label{approx:est-2}
   \|\compU-\compUP\|_{\mu,h}\preceq &\;h^{k+1}\left(\mu^{1/2}\|\bld u\|_{k+2}+
   \frac{\alpha}{\lambda^{1/2}}\|p\|_{k+1}\right),\\
  \label{approx:est-3}
   \|\divs(\bld u -\Pi\bld u)\|\preceq &\;h^{k+1}\left(\frac{\mu^{1/2}}{\lambda^{1/2}}\|\bld u\|_{k+2}+
   \|\divs\,\bld u\|_{k+1}+
   \frac{\alpha}{\lambda}\|p\|_{k+1}\right). 
  \end{align}
 \end{subequations}
\end{lemma}

\begin{proof}
 The pressure estimate follows from \cite{Oikawa:15}.
The displacement estimates follow from \cite{FuLehrenfeld18a}.
In particular, we introduce $\compVh:=(\Pi_V \bld u, \PiVF \bld u^t)$ where $\Pi_V$ is the classical BDM interpolator, 
\cite[Proposition 2.3.2]{BoffiBrezziFortin13}, and estimate the error by first applying a triangle inequality to split
\[
 \|\compU-\compUP\|_{\mu,h} = 
  \|\compVh-\compU\|_{\mu,h}+ \|\compUP-\compVh\|_{\mu,h}.
\]
Using coercivity result in Lemma \ref{lemma:coercivity}, we get
\begin{align*}
&  \|\compUP-\compVh\|_{\mu,h}^2
  +\lambda \|\mathrm{div}(\Pi \bld u-\bld v_h)\|^2 \preceq\; b_h(\compUP-\compVh, \compUP - \compVh)\\
&\hspace{2cm}  =\;  b_h(\compU-\compVh, \compUP - \compVh) + (\Pi p-p,\alpha\,\divs(\Pi \bld u-\bld v_h))\\
&\hspace{2cm}  \preceq\;  
\|\compU-\compVh\|_{\mu,*,h}\|\compUP-\compVh\|_{\mu,h}
+\alpha \|\Pi p-p\| \|\divs(\Pi \bld u-\bld v_h)\|.
\end{align*}
Hence, applying the triangle inequality,
\begin{align*}
 \|\compU-\compUP\|_{\mu,h}\preceq &\;
  \|\compU-\compVh\|_{\mu,*,h}+\frac{\alpha}{\lambda^{1/2}} \|p-\Pi p\|,\\
 \|\divs(\bld u-\compUP)\|\preceq&\; 
  \frac{1}{\lambda^{1/2}}\|\compU-\compVh\|_{\mu,*,h}
  + \|\divs(\bld u-\Pi_V\bld u)\|
  +\frac{\alpha}{\lambda} \|p-\Pi p\|.
\end{align*}
The estimates \eqref{approx:est-2}, \eqref{approx:est-3} now follows from the standard approximation 
properties of the BDM interpolator $\Pi_V$.
\end{proof}

To further simplify notation, we denote 
\begin{align}
\label{notation}
 \compeu = \compUh-\compUP, \quad
 \compep = \compPh-\compPP, \quad\du = \bld u-\Pi\bld u, \quad\dpp = p-\Pi p.
\end{align}

Combining the numerical scheme \eqref{scheme} with the consistency result in Lemma \ref{lemma:consistency}, 
adding and subtracting the above elliptic projectors, we arrive at the following error equations:
\begin{subequations}
\label{error-eq}
\begin{align}
\label{error-eq-1}
 (c_s \dot{\ep}+\alpha\,\divs(\dot{\eu}), w_h)
+
a_h(\compep, \compW_h) =&\;
 (c_s \dot{\dpp}+\alpha\,\divs(\dot{\du}), w_h),\\
\label{error-eq-2}
b_h(\compeu, \compVh)
-(\ep,\alpha\,\divs(\bld v_h))=&\; 0,
\end{align}
\end{subequations}
for all $\compW_h=(w_h,\wf_h) \in \compPhspace$ and 
$\compVh=(\bld v_h,\vf_h) \in \compUhspace$.

By the inf-sup stability \cite{BoffiBrezziFortin13} of the finite elements pair 
$\Wh\times \Vh\subset L^2(\Omega)\times H_0(\mathrm{div},\Omega)$, we have the following 
pressure estimate.
\begin{lemma}
 \label{lemma:inf-sup}
 Let $\overline{\ep}$ be the average of $\ep$ on $\Omega$. Then, we have 
 \[
  \alpha \|\ep-\overline{\ep}\|\preceq \mu^{1/2}\|\compeu\|_{\mu,h}+\lambda \|\divs\,\eu\|.
 \]
\end{lemma}
\begin{proof}
 By inf-sup stability \cite{BoffiBrezziFortin13}, 
 there exists a function $\compWW_h=(\bld w_h,\wwf_h)\in \compUhspace$ such that 
 \[
  \divs \,\bld w_h = \ep-\overline{\ep}, \quad \text{ and }
  \|\compWW_h\|_{\mu,h}\le \mu^{1/2}\|\ep-\overline{\ep}\|.
 \]
The estimate in Lemma \ref{lemma:inf-sup} 
follows directly by taking $\compVh = \compWW_h$ in \eqref{error-eq-2}, 
using the fact that $(\divs \,\bld w_h,\overline{\ep}) = 0$,
and 
applying the Cauchy-Schwarz inequality.
\end{proof}

Now, we are ready to present our main results on the semi-discrete error estimates.
\begin{theorem}
\label{thm:energy-semi}
Let the stabilization parameter $\tau_0$ be sufficiently large. 
Let $(\compPh,\compUh)\in \compPhspace\times\compUhspace$ be the solution to \eqref{scheme} 
with initial data 
$ \compPh(0)=\compPP(0)$ and $ \compUh(0) = \compUP(0)$.
 Then, the following estimate holds for all $T>0$:
 \begin{alignat}{2}
 \label{est-1}
\vertiii{\{\compep(T),\compeu(T)\}}_{h}^2
+\int_0^Ta_h(\compep,\compep)\,\mathrm{dt}
\preceq h^{2k+2}\,\Xi_1,
 \end{alignat}
 where 
 \[
  \Xi_1 = T\,\int_0^T (c_s+\frac{\alpha^2(\lambda+\mu)}{\lambda^2})
  \|\dot{p}\|_{k+1}^2
+ \frac{\mu(\lambda+\mu)}{\lambda}\|\dot{\bld u}\|_{k+2}^2
+ (\lambda+\mu)\|\divs\,\dot{\bld u}\|_{k+1}^2
\,\mathrm{dt}.
 \]

\end{theorem}

\begin{remark}[Robust displacement estimate]
The above estimate for the displacement is robust with respect to the incompressible limit 
$c_s\rightarrow 0$ and 
$\lambda\rightarrow +\infty$, as long as 
the  term  $\lambda\|\divs\,\dot{\bld u}\|_{k+1}^2$ is bounded.
 It is also robust in the degenerate case as the  permeability $\kappa\rightarrow 0$.
\end{remark}

\begin{proof}
We use a standard energy argument.
Taking $\compW_h = \compep$ and $\compVh = \dot{\compeu}$ in the error equations \eqref{error-eq} and adding, 
we get 
\begin{align*}
  (c_s \dot{\ep}, \ep)
+
b_h(\compeu, \dot{\compeu})+
a_h(\compep, \compep) =&\;
 (c_s \dot{\dpp}+\alpha\,\divs(\dot{\du}), \ep)\\ 
 =&\; 
  (c_s \dot{\dpp}, \ep)
  + (\alpha\,\divs(\dot{\du}), \ep-\overline{\ep})
\end{align*}
Applying the Cauchy-Schwarz inequality on the above right hand side and using 
the estimate in Lemma \ref{lemma:inf-sup},
we have 
\begin{align*}
  (c_s \dot{\dpp}, \ep)
  + (\alpha\,\divs(\dot{\du}), \ep-\overline{\ep})
  \le&\; 
  {c_s\|\dot{\dpp}\|\|\ep\|
  +
 \|\divs\,\dot{\du}\|\,(\mu^{1/2}\|\eu\|_{\mu,h}+\lambda\|\divs\,\eu\|})\\
 \le \;
\Big(\underbrace{c_s\|\dot{\dpp}\|^2+(\mu+\lambda)\|\divs\,\dot{\du}\|^2}_{:=\Theta}&\Big)^{1/2}
\left( c_s\|\ep\|^2
+
\|\eu\|_{\mu,h}^2+\lambda\|\divs\,\eu\|^2
\right)^{1/2}. 
\end{align*}
Combing this estimate with 
the above identity, and invoking the coercivity result \eqref{coercivity-2}, we get 
\[
\frac{1}{2} \partial_t\Big(  c_s(\ep, \ep)
+
b_h(\compeu, {\compeu})\Big)+
a_h(\compep, \compep) \preceq \Theta^{1/2} \Big(  c_s(\ep, \ep)
+
b_h(\compeu, {\compeu})\Big)^{1/2}
\]
Recalling that $\compep(0) = 0$ and $\compeu(0) = 0$, then an application of the 
Gronwall's inequality implies that 
\[
 c_s(\ep(T), \ep(T))
+
b_h(\compeu(T), {\compeu}(T))+\int_0^Ta_h(\compep,\compep)\,\mathrm{dt}
\preceq
T\,\int_0^T\Theta \,\mathrm{dt},
\]
for all $T>0$. 
Combining the above estimate with \eqref{coercivity-2} and \eqref{approx:est}, we get the 
desired inequality in 
Theorem \ref{thm:energy-semi}.
\end{proof}

\begin{corollary}
 \label{coro:ener}
 Let assumptions of Theorem \ref{thm:energy-semi} holds. 
 Then, the following estimate holds for all $T>0$:
 \begin{alignat}{2}
 \label{est-2}
 \vertiii{\{\dot\compep(T),\dot\compeu(T)\}}_{h}^2
\preceq h^{2k+2}\,\Xi_2,
 \end{alignat}
 where 
 \[
  \Xi_2 = T\,\int_0^T (c_s+\frac{\alpha^2(\lambda+\mu)}{\lambda^2})
  \|\ddot{p}\|_{k+1}^2
+ \frac{\mu(\lambda+\mu)}{\lambda}\|\ddot{\bld u}\|_{k+2}^2
+ (\lambda+\mu)\|\divs\,\ddot{\bld u}\|_{k+1}^2
\,\mathrm{dt}.
 \]
\end{corollary}
\begin{proof}
 Take one time derivative of the error equations \eqref{error-eq}. Then proceed as in the 
 proof of Theorem \ref{thm:energy-semi}.
\end{proof}

Now, we give a robust pressure estimate, with respect to $c_s$,
under the assumption that permeability $\kappa$ is away from 
{\it zero}.

\begin{theorem}
 \label{thm:pressure}
 Let the assumptions of Theorem \ref{thm:energy-semi} hold.
 Then, for all $T>0$, the following estimate holds
 \[
  \kappa\|\ep(T)\|\preceq  \kappa\|\compep(T)\|_{1,h}\preceq 
  h^{k+1}((c_s^{1/2}+\frac{\alpha}{{\lambda}^{1/2}})\Xi_2^{1/2}+\Xi_3),
 \]
 where $\Xi_2$ is given in Corollary \ref{coro:ener}, and
 \[
  \Xi_3 = (c_s+ \frac{\alpha^2}{\lambda})\|\dot{p}(T)\|_{k+1}
+ 
\alpha\left(\frac{\mu^{1/2}}{\lambda^{1/2}}\|\dot{\bld u}(T)\|_{k+2}+
   \|\divs\,\dot{\bld u}(T)\|_{k+1}\right).  
 \]
\end{theorem}
\begin{proof}
 Taking $\compW_h = \compep$ in \eqref{error-eq-1},
 reordering terms, 
 and applying the Cauchy-Schwarz inequality,
 we have
 \begin{align*}
  a_h(\compep,\compep) = &\;
  \left(c_s (\dot{\dpp}-\dot{\ep})+\alpha\,\divs(\dot{\du}-\dot{\eu}), \ep\right)\\
 \preceq &\; \left(c_s(\|\dot{\dpp}\|+\|\dot{\ep}\|)
 +\alpha(\|\divs\,\dot{\du}\|+\|\
 \divs\,\dot{\eu}\|)
 \right)\|\ep\|
 \end{align*}
 Invoking the discrete Poincar\'e inequality
 \cite{DiPietroErn10},
 $\|w_h\|\preceq \|\compW_h\|_{1,h}$ for
 all $\compW_h\in \compPhspace$, and using the coercivity result \eqref{coercivity-1}, we get
 \[
\kappa\|\compep\|_{1,h} \preceq \; c_s(\|\dot{\dpp}\|+\|\dot{\ep}\|)
 +\alpha(\|\divs\,\dot{\du}\|+\|\
 \divs\,\dot{\eu}\|).
 \]
Combing the above estimate with Lemma \ref{lemma:approx} and Corollary \ref{coro:ener}, we get the 
desired inequality in Theorem \ref{thm:pressure}.
\end{proof}

We conclude this section with a remark on (slightly) relaxing the $H(\mathrm{div})$-conformity of the displacement space to reduce global coupling.
\begin{remark}[Relaxed $H(\mathrm{div})$-conformity]
\label{rk:relax}
We noticed that to reach a convergence rate of ${k+1}$ for the ``energy norm'' $\vertiii{\cdot}_h$, we need unknowns of polynomial degree
$k+1$ on the facets. We follow the idea of \cite{Lederer17} to relax the highest-order normal conformity of the displacement space:
\[
 \Vh^- : =\; \{\bld v\in \prod_{T\in\Oh}[\pol^{k+1}(T)]^d, \;\;
\Pi_F^k\jmp{\bld v\cdot\bld n}_F = 0 \,\forall F\in\Eh\}\subset H_0(\mathrm{div},\Omega),
\]
where $\Pi_F^k: L^2(F)\rightarrow \pol^k(F)$ is the $L^2$-projection. The resulting semi-discrete scheme still use the formulation \eqref{scheme}, but with 
the space $\compUhspace^-:= \Vh^- \times \widehat{\bld V}_h$ for displacement and $\compPhspace$ for pressure. 
The globally coupled degrees of freedom (after static condensation) for this modification consists of 
polynomials of degree $k$ for the displacement and polynomials of degree $k-1$ for the pressure per facet; while the 
that for the original scheme 
consists of 
polynomials of degree $k+1$ for the normal-component of the displacement, 
polynomials of degree $k$ for the tangential-component of the displacement, 
and polynomials of degree $k-1$ for the pressure per facet.

We present numerical results in Section \ref{sec:numerics} to validate the optimality of such modification, 
and refer interested reader to \cite{Lederer17,FuLehrenfeld18a} for the analysis.
\end{remark}

\section{Fully-discrete Scheme}
\label{sec3:time}
For the temporal discretization 
of the semi-discrete DAE \eqref{scheme}, we consider the $m$-step 
BDF \cite[Chapter V]{HairerWanner10} method with step size $\Delta t >0$: 
for $n\ge m$, find $(\compPh^n, \compUh^n)\in \compPhspace\times \compUhspace$ such that
\begin{subequations}
\label{bdf-time}
\begin{align}
\label{bdf-1}
\sum_{j=0}^m\frac{\delta_j}{\Delta t}(c_s {p}_{h}^{n-j}+\alpha\,\divs({\bld u}_h^{n-j}), w_h)
+
a_h(\compPh^n, \compW_h) =&\; (f(t^n), w_h),\\
\label{bdf-2}
b_h(\compUh^n, \compVh)
-(p_h^n,\alpha\,\divs(\bld v_h))=&\; (\bld g(t^n), \bld v_h),
\end{align}
\end{subequations}
for all $(\compW_h,\compVh)\in \compPhspace\times \compUhspace$
with given starting values $\{\compPh^i, \compUh^i\}_{i=0}^{m-1}$, where $t^n =n \Delta t$.
The method coefficients $\delta_j$ are determined
from the relation 
\begin{align}
\label{zeta-f}
 \delta(\zeta) = \sum_{j=0}^m \delta_j \zeta^j = \sum_{\ell=1}^m \frac{1}{\ell}(1-\zeta)^\ell. 
\end{align}
The BDF method is known to have order $m$ for $m\le 6$, 
and is A-stable for $m=1$ and $m=2$, but not for $m\ge 3$. 


Next, we provide error estimates for the 
fully discrete scheme \eqref{bdf-time} with $m=2$ using an energy argument. 
We remark that the analysis for the cases with $3\le m\le 5$ is similar but more technical as 
one needs to use the multiplier technique \cite{NevanlinnaOdeh81, Akrivis15}.

To simplify notation, we denote the backward difference operator 
\begin{align}
 \label{bdf2}
 \mathsf{d_t}\phi^n := \frac{3\phi^{n}-4\phi^{n-1}+\phi^{n-2}}{2\Delta t}.
\end{align}
Let $(\cdot,\cdot)$ be an inner product with associated norm $|\cdot|$. Then,
a straightfoward calculation yields
\begin{align}
\label{energy-x}
 (\mathsf{d_t}\phi^n, \phi^n) =&\; \frac{1}{4\Delta t}\Big(
 |\phi^n|^2+|2\phi^n-\phi^{n-1}|^2
 -|\phi^{n-1}|^2-|2\phi^{n-1}-\phi^{n-2}|^2\\
&\;\quad\quad\;\; +|\phi^n-2\phi^{n-1}+\phi^{n-2}|^2\Big)\nonumber
\end{align}

We continue to use the notation \eqref{notation}. 
Denoting the norm
\begin{align}
\label{norm-sup}
 \|\phi\|_{L^\infty(H^s)}:= \sup_t \|\phi(t)\|_{H^s(\Omega)},
\end{align}
we have the following result on the consistency of the scheme \eqref{bdf-time}.
\begin{lemma}
\label{lemma:cs}
Let $(\compPh^n, \compUh^n)\in \compPhspace\times \compUhspace$, $n\ge 2$,
be the solution to equations \eqref{bdf-time} with $m=2$
and starting values $(\compPh^0, \compUh^0)$ and $(\compPh^1, \compUh^1)$.
Let $p^n:=p(t^n)$ and $\bld u^n :=\bld u(t^n)$ be the exact solution to equations \eqref{eqns}
at time $t^n$. Then, there holds, for $n\ge 2$,
\begin{subequations}
\label{f-error-eq}
\begin{align}
\label{f-error-eq-1}
 (c_s \mathsf{d_t}{\ep^n}+\alpha\,\divs(\mathsf{d_t}{\eu^n}), w_h)
+
a_h(\compep^n, \compW_h) =&\;
 \mathcal{E}_h^n(w_h)\\
\label{f-error-eq-2}
b_h(\compeu^n, \compVh)
-(\ep^n,\alpha\,\divs(\bld v_h))=&\; 0,
\end{align}
\end{subequations}
for all $\compW_h=(w_h,\wf_h) \in \compPhspace$ and 
$\compVh=(\bld v_h,\vf_h) \in \compUhspace$, 
where 
\[
 \mathcal{E}_h^n(w_h) : = 
c_s\left(\mathsf{d_t}{\dpp^n}- \mathsf{d_t}{p^n}+\dot{p}^n,w_h\right)
+
\alpha\left(\divs(\mathsf{d_t}{\du^n}-\mathsf{d_t}{\bld u}^n+
 \dot{\bld u}^n), w_h-\overline{w}_h\right),
\]
and $\overline{w}_h$ is the average of $w_h$ on $\Omega$.
Moreover, there holds
\begin{align}
\label{cst-est}
 \| \mathcal{E}_h^n(w_h)\|\preceq &\; 
 c_s\,\mathcal{O}_{2,I} \|w_h\|
 +\alpha\, \mathcal{O}_{2,I\!I} \|w_h-\overline{w}_h\|,
\end{align}
where, for integer $s\ge 1$, 
\begin{align*}
 \mathcal{O}_{s,I} :=&\;h^{k+1}
 \|\frac{\partial p}{\partial t}\|_{L^\infty(H^{k+1})}
 +\Delta t^s \|\frac{\partial^{s+1} p}{\partial t^{s+1}}\|_{L^\infty(L^2)}\\
 \mathcal{O}_{s,I\!I} :=&\;h^{k+1}\left(
\frac{\mu^{1/2}}{\lambda^{1/2}}\|\frac{\partial \bld u}{\partial t}\|_{L^\infty(H^{k+2})}+
   \|\divs\,\frac{\partial \bld u}{\partial t}\|_{L^\infty(H^{k+1})}+
   \frac{\alpha}{\lambda}\|\frac{\partial p}{\partial t}\|_{L^\infty(H^{k+1})} 
\right)\\
&\;+\Delta t^s \|\divs\, \frac{\partial^{s+1} \bld u}{\partial t^{s+1}}\|_{L^\infty(L^2)}.
\end{align*}
\end{lemma}
\begin{proof}
 The error equations \eqref{f-error-eq} follows from the scheme \eqref{bdf-time} and 
 the consistency result in Lemma \ref{lemma:consistency}.
 The estimate \eqref{cst-est} follows from the Cauchy-Schwarz inequality, 
 the approximation properties in Lemma \ref{lemma:approx} of the elliptic projector, 
 and Taylor expansion in time.
\end{proof}

Our main result on the fully-discrete error estimates is given below.
\begin{theorem}
 \label{thm:energy-full}
Let $(\compPh^n, \compUh^n)\in \compPhspace\times \compUhspace$, $n\ge 2$,
be the solution to equations \eqref{bdf-time} with $m=2$
and starting values $(\compPh^0, \compUh^0)$ and $(\compPh^1, \compUh^1)$.
Let $p^n:=p(t^n)$ and $\bld u^n :=\bld u(t^n)$ be the exact solution 
to equations \eqref{eqns} at time $t^n$. 
Then, there holds, for $N\ge 2$,
\begin{align}\label{full-est}
\vertiii{\{\compep^N,\compeu^N\}}_{h}^2
+\Delta t\sum_{n=2}^Na_h(\compep^n,\compep^n)
\preceq 
\exp(N\Delta t)&\;\Big(
\sum_{i=0}^1\vertiii{\{\compep^i,\compeu^i\}}_{h}^2\\
&\hspace{-0.2cm}+N\Delta t(c_s\mathcal{O}_{2,I}^2+(\lambda+\mu)\mathcal{O}_{2,I\!I}^2)\Big)\nonumber
\end{align}
\end{theorem}

\begin{proof}
Taking $\compW_h = \compep^n$ in equation \eqref{f-error-eq-1} and 
$\compVh = \mathsf{d_t}\compeu^n$ in equation \eqref{f-error-eq-2}, and adding and summing
the resulting expression for $n = 2,\cdots, N$, we get
\begin{align}
\label{haha}
 \sum_{n=2}^N (c_s \mathsf{d_t}{\ep^n}, \ep^n)
 +b_h(\compeu^n,\mathsf{d_t}{\compeu^n}) + a_h(\compep^n,\compep^n)
 =&\sum_{n=2}^N \mathcal{E}_h^n(\compep^n)\nonumber\\
 &\hspace{-2.5cm}\preceq 
\sum_{n=2}^N (c_s\mathcal{O}_{2,I}^2+(\lambda+\mu)\mathcal{O}_{2,I\!I}^2)^{1/2}
\vertiii{\{\compep^n, \compeu^n\}}_{h},
\end{align}
where we used a combination of Lemma \ref{lemma:cs} and Lemma \ref{lemma:inf-sup} to derive the 
last inequality.

The identity \eqref{energy-x} implies that 
\begin{align*}
   \sum_{n=2}^N (c_s \mathsf{d_t}{\ep^n}, \ep^n)
 +b_h(\compeu^n,\mathsf{d_t}{\compeu^n})
 \ge &\;
\frac{1}{4\Delta t}\Big( c_s\|{\ep^N}\|^2
 +b_h(\compeu^N,{\compeu^N})
 -
 c_s\|{\ep^1}\|^2\\
 &\hspace{-2cm}-c_s\|2{\ep^1}-\ep^0\|^2)
 -b_h(\compeu^1,{\compeu^1})
 -b_h(2\compeu^1-\compeu^0,{2\compeu^1-\compeu^0})
\Big).
\end{align*}
Hence, 
\[
\frac{1}{\Delta t}(\vertiii{\{\compep^N, \compeu^N\}}_{h}^2
-\sum_{i=0}^1\vertiii{\{\compep^i, \compeu^i\}}_{h}^2)
\preceq \sum_{n=2}^N (c_s \mathsf{d_t}{\ep^n}, \ep^n)
 +b_h(\compeu^n,\mathsf{d_t}{\compeu^n})
\]
Combining this estimate with \eqref{haha}, we get 
\begin{align*}
 \vertiii{\{\compep^N, \compeu^N\}}_{h}^2
 +\Delta t\sum_{n=2}^Na_h(\compep^n,\compep^n)
 \preceq &
\Delta t \sum_{n=2}^N (c_s\mathcal{O}_{2,I}^2+(\lambda+\mu)\mathcal{O}_{2,I\!I}^2)^{1/2}
\vertiii{\{\compep^n, \compeu^n\}}_{h}\\
&+\sum_{i=0}^1\vertiii{\{\compep^i, \compeu^i\}}_{h}^2.
\end{align*}
Finally, the estimate \eqref{full-est} follows from a discrete Gronwall's inequality, c.f. 
\cite[Lemma 5.1]{HeywoodRannacher90}.
\end{proof}

\begin{remark}[Higher order BDF method]
\label{rk:2}
For $m$-step BDF methods with $m = 1$ or $3\le m\le 5$, 
we can 
still use a similar energy argument 
to 
derive the following estimate
\begin{align*}
 \vertiii{\{\compep^N,\compeu^N\}}_{h}^2
+\Delta t\sum_{n=2}^Na_h(\compep^n,\compep^n)
\preceq 
\exp(N\Delta t)&\;\Big(
\sum_{i=0}^{m-1}\vertiii{\{\compep^i,\compeu^i\}}_{h}^2\\
&\hspace{-0.2cm}+N\Delta t(c_s{\mathcal{O}}_{m,I}^2+(\lambda+\mu){\mathcal{O}}_{m,I\!I}^2)\Big).
\end{align*}
In the cases for $3\le m\le 5$, we need to apply 
the multiplier technique \cite{NevanlinnaOdeh81}, and take in the energy argument
the test function in the error equation 
\eqref{f-error-eq-1} to be $\compW_h:= \compep^n-\eta\, \compep^{n-1}$ with the multiplier 
$\eta = 0.0836$ for $m=3$,
$\eta = 0.2878$ for $m=4$, and
$\eta = 0.8160$ for $m=5$. More details of the multiplier technique can be found in the recent 
publications \cite{LubichMansour13, Akrivis15}.
\end{remark}

\begin{remark}[Starting values for BDF2 and BDF3]
The $m$-step BDF method needs ${m-1}$ starting values to begin with. 

For BDF2, we can simply take 
$(\compPh^1, \compUh^1)$ to be the Backward Euler solution to equations \eqref{bdf-time} with $m=1$
and $(\compPh^0, \compUh^0) = (\compPP(0), \compUP(0))$.
This implies 
\[
\vertiii{\{\ep^0,\eu^0\}}_{h}=0, \quad\quad
\vertiii{\{\ep^1,\eu^1\}}_{h}\preceq 
\Delta t (c_s{\mathcal{O}}_{1,I}^2+(\lambda+\mu){\mathcal{O}}_{1,I\!I}^2).
\]
Combining these estimates with \eqref{full-est}, we readily have 
$\vertiii{\{\ep^N,\eu^N\}}_{h}$ converges $k+1$-th order in space, 
and second-order in time.

For BDF3, we take 
$(\compPh^0, \compUh^0) = (\compPP(0), \compUP(0))$, 
$(\compPh^i, \compUh^i)$, $i=1,2$, to be the solution with Crank-Nicolson time stepping.
Similarly, the local error for 
$\vertiii{\{\ep^i,\eu^i\}}_{h}$, $i=1,2$, are 
third-order in time.
Then, the estimates in Remark \ref{rk:2} yields that
$\vertiii{\{\ep^N,\eu^N\}}_{h}$ converges $k+1$-th order in space, 
and third-order in time.
\end{remark}

\begin{remark}[diagonally implicit Runge-Kutta time stepping]
Alternatively, we can apply 
the (one-step, multi-stage) diagonally implicit Runge-Kutta (DIRK) methods to solve the DAE \eqref{scheme}.
We refer the interested reader to the references \cite{NguyenPeraireCockburn11b,JaustSchutz14} for a setup.
However, in our numerical experiments not documented here, 
we do observe the order reduction \cite{CarpenterGottliebAbarbanelDon95,RosalesSeiboldShirokoffZhou17} 
in high-order ($\ge3$) DIRK schemes due to inappropriate boundary treatment in the intermediate stages. 
We only observe second order accuracy for third- and fourth-order DIRK schemes. 
\end{remark}


\section{Numerical results}
\label{sec:numerics}
In this section, we present several numerical experiments to illustrate the performance of the proposed method.
The numerical results are performed using the NGSolve software \cite{Schoberl16}.

\subsection{Accuracy for a smooth solution with a large $\lambda$}
In order to confirm the optimal convergence rates in Section \ref{sec2:disc} and Section \ref{sec3:time}, 
we consider a manufactured smooth exact solution, similar to the one considered in \cite[Section 7.1]{Y117}. 
Specifically, we take the domain to be $\Omega=(0,1)^2$, with the exact 
displacement $\bld u=(u, v)$ and exact pressure $p$ given by 
\begin{align*}
 u(\bld x, t) =& -e^{-t}\cos(\pi x)\sin(\pi y)+\frac{1}{\mu+\lambda}e^{-t}
 \sin(\pi x)\sin(\pi y)\\
 v(\bld x, t) =& e^{-t}\sin(\pi x)\cos(\pi y)+\frac{1}{\mu+\lambda}e^{-t}
 \sin(\pi x)\sin(\pi y)\\
 p(\bld x, t) =&  e^{-t}\sin(\pi x)\sin(\pi y).
\end{align*}
Note that the solution is designed to satisfy \[\divs\,\bld u = \pi e^{-t}
 \sin(\pi (x+y))/(\mu+\lambda)\rightarrow 0 \quad \text{as }\quad \lambda\rightarrow +\infty.\]
We impose Dirichlet boundary conditions for both $\bld u$ and $p$, and choose the following material parameters:
\[
 c_0=0, \quad \alpha = 1,\quad \kappa = 1,\quad \lambda =10^5,\quad \mu = 1.
\]
The final computational time is $T=0.5$.

Our computation is based on uniform triangular meshes; see Figure \ref{fig:1} for the coarsest mesh with mesh size $h=1/4$. 
We consider the fully discrete scheme \eqref{bdf-time},  
with the (spatial) polynomial degree $k$ in the finite element spaces \eqref{space} varying from $1$ to $3$,
and the (temporal) BDF3 method ($m=3$). 
We also present numerical results using the relaxed $H(\mathrm{div})$-conformity approach, c.f. Remark \ref{rk:relax}. 
The stabilization parameter $\tau$ in the bilinear forms \eqref{bilinearforms}
is taking to be $\tau = 10 k^2$ for all the tests.
We take the time step size to be $\Delta t = h^{\max\{(k+1)/3,1\}}$, where $h$ is the 
spatial mesh size.  The error in the norm $\vertiii{\cdot}_h$, and the $L^2$-norms for displacement and pressure at the final time 
$T=0.5$ are 
recorded in Table \ref{table:1} on a sequences of uniformly refined meshes for the original scheme \ref{scheme}, and 
in Table \ref{table:2} for the relaxed $H(\mathrm{div})$-conforming scheme, c.f. Remark \ref{rk:relax}. 
In both tables, we observe the optimal 
convergence rates for the norm $\vertiii{\cdot}_h$, 
in full agreements with our main result in Theorem \ref{thm:energy-full} and  Remark \ref{rk:2}, 
we also observe optimal convergence rates  in the $L^2$-norm of the displacement ($k+2$), and the $L^2$-norm of the 
pressure ($k+1$).

\begin{figure}[ht!]
 \caption{The coarsest mesh with $h = 1/4$}
 \includegraphics[width=.4\textwidth]{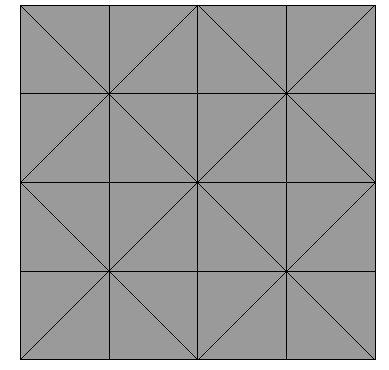}
\label{fig:1}
 \end{figure}

 \begin{table}[ht!]
\caption{Convergence study at the final time $T=0.5$: The original scheme.} 
\centering 
{%
\begin{tabular}{ c c cc cc c c}
\toprule
 &  mesh & \multicolumn{2}{c}{$\vertiii{\{\compP, \compU\}-
  \{\compPh, \compUh\}}_{h}$}
          & \multicolumn{2}{c}{$\|\bld u -\bld u_{h}\|_{\Omega}$}
          &\multicolumn{2}{c}{{$\|p-p_h\|_{\Omega}$}}\\
\midrule
$k$ & $h$ & error & order & error & order & error & order
\tabularnewline
\midrule
\multirow{5}{2mm}{1} 
&1/4& 7.214e-02  &  -   &  3.589e-03  &  -    &  2.190e-02  &  - \\
& 1/8&1.854e-02  &  1.96    &  4.476e-04  &  3.00    &  5.194e-03  &  2.08    \\
&1/16& 4.692e-03  &  1.98    &  5.558e-05  &  3.01    &  1.304e-03  &  1.99    \\ 
&1/32& 1.178e-03  &  1.99    &  6.911e-06  &  3.01    &  3.263e-04  &  2.00    \\ 
&1/64& 2.951e-04  &  2.00    &  8.609e-07  &  3.00    &  8.159e-05  &  2.00    \\ 
\midrule
\multirow{5}{2mm}{2} 
&1/4& 8.342e-03  &  -    &  3.306e-04  &  -    &  1.832e-03  &  -    \\ 
&1/8& 1.034e-03  &  3.01    &  2.024e-05  &  4.03    &  2.421e-04  &  2.92    \\ 
&1/16& 1.283e-04  &  3.01    &  1.239e-06  &  4.03    &  3.037e-05  &  3.00    \\ 
&1/32& 1.598e-05  &  3.01    &  7.658e-08  &  4.02    &  3.799e-06  &  3.00    \\ 
&1/64& 1.994e-06  &  3.00    &  4.759e-09  &  4.01    &  4.750e-07  &  3.00    \\ 
\midrule
\multirow{5}{2mm}{3} 
& 1/4&
 7.739e-04  &  -    &  2.583e-05  & -    &  1.851e-04  &  -    \\ 
& 1/8&
 4.785e-05  &  4.02    &  7.659e-07  &  5.08    &  1.130e-05  &  4.03    \\ 
&1/16&
 3.019e-06  &  3.99    &  2.357e-08  &  5.02    &  7.095e-07  &  3.99    \\ 
&1/32&
 1.879e-07  &  4.01    &  7.244e-10  &  5.02    &  4.409e-08  &  4.01    \\ 
&1/64&
 1.178e-08  &  4.00    &  4.849e-11  &  3.90    &  2.761e-09  &  4.00    \\ 
\bottomrule
\end{tabular}}
\label{table:1} 
\end{table}

 \begin{table}[ht!]
\caption{Convergence study at the final time $T=0.5$: The relaxed $H(\mathrm{div})$-conforming scheme. } 
\centering 
{%
\begin{tabular}{ c c cc cc c c}
\toprule
 &  mesh & \multicolumn{2}{c}{$\vertiii{\{\compP, \compU\}-
  \{\compPh, \compUh\}}_{h}$}
          & \multicolumn{2}{c}{$\|\bld u -\bld u_{h}\|_{\Omega}$}
          &\multicolumn{2}{c}{{$\|p-p_h\|_{\Omega}$}}\\
\midrule
$k$ & $h$ & error & order & error & order & error & order
\tabularnewline
\midrule
\multirow{5}{2mm}{1} 
&1/4&
 6.201e-02  &  -    &  3.441e-03  &  -    &  2.190e-02  &  -    \\ 
& 1/8&
 1.641e-02  &  1.92    &  4.691e-04  &  2.87    &  5.194e-03  &  2.08    \\ 
&1/16& 
 4.262e-03  &  1.94    &  6.326e-05  &  2.89    &  1.304e-03  &  1.99    \\ 
&1/32&
 1.085e-03  &  1.97    &  8.237e-06  &  2.94    &  3.263e-04  &  2.00    \\ 
&1/64&
 2.732e-04  &  1.99    &  1.048e-06  &  2.97    &  8.159e-05  &  2.00    \\ 
\midrule
\multirow{5}{2mm}{2} 
&1/4&
 8.170e-03  &  -    &  3.253e-04  &  -    &  1.832e-03  &  --    \\ 
&1/8&
 1.036e-03  &  2.98    &  2.040e-05  &  4.00    &  2.421e-04  &  2.92    \\ 
&1/16&
 1.295e-04  &  3.00    &  1.271e-06  &  4.01    &  3.037e-05  &  3.00    \\ 
&1/32& 
 1.618e-05  &  3.00    &  7.921e-08  &  4.00    &  3.799e-06  &  3.00    \\ 
&1/64& 
 2.021e-06  &  3.00    &  4.943e-09  &  4.00    &  4.750e-07  &  3.00    \\ 
\midrule
\multirow{5}{2mm}{3} 
& 1/4&
 7.633e-04  &  -    &  2.562e-05  &  -    &  1.851e-04  &  -    \\ 
& 1/8&
 4.741e-05  &  4.01    &  7.749e-07  &  5.05    &  1.130e-05  &  4.03    \\ 
&1/16&
 3.024e-06  &  3.97    &  2.442e-08  &  4.99    &  7.095e-07  &  3.99    \\ 
&1/32&
 1.894e-07  &  4.00    &  7.615e-10  &  5.00    &  4.409e-08  &  4.01    \\ 
&1/64&
 1.191e-08  &  3.99    &  3.930e-11  &  4.28    &  2.761e-09  &  4.00    \\ 
\bottomrule
\end{tabular}}
\label{table:2} 
\end{table}

\subsection{Barry and Mercer's problem}
We consider the Barry and Mercer's problem \cite{BarryMercer99}, for which an exact
solution is available in terms of infinite series 
(we refer the reader to the cited paper and also to \cite[Section 4.2.1]{Phillips05} for the expression). 
It models the behavior of a rectangular uniform porous material with a pulsating point source, drained on all sides, and on
which zero tangential displacements are assumed on the whole boundary. The point-source corresponds to a sine wave on the 
rectangular domain $(0,a)\times (0,b)$ and is given as \[f(t) = 2\beta \delta_{\bld x_0}\sin(\beta t),\] where 
$\beta = \frac{(\lambda+2\mu)\kappa}{ab}$ and $\delta_{\bld x_0}$ is the Dirac delta at the point $\bld x_0$.
The computational domain together with the boundary conditions are depicted in Figure \ref{fig:2}.

\begin{figure}[ht!]
 \caption{Computational domain and boundary conditions for the Barry and Mercer's problem}
 \includegraphics[width=0.8\textwidth]{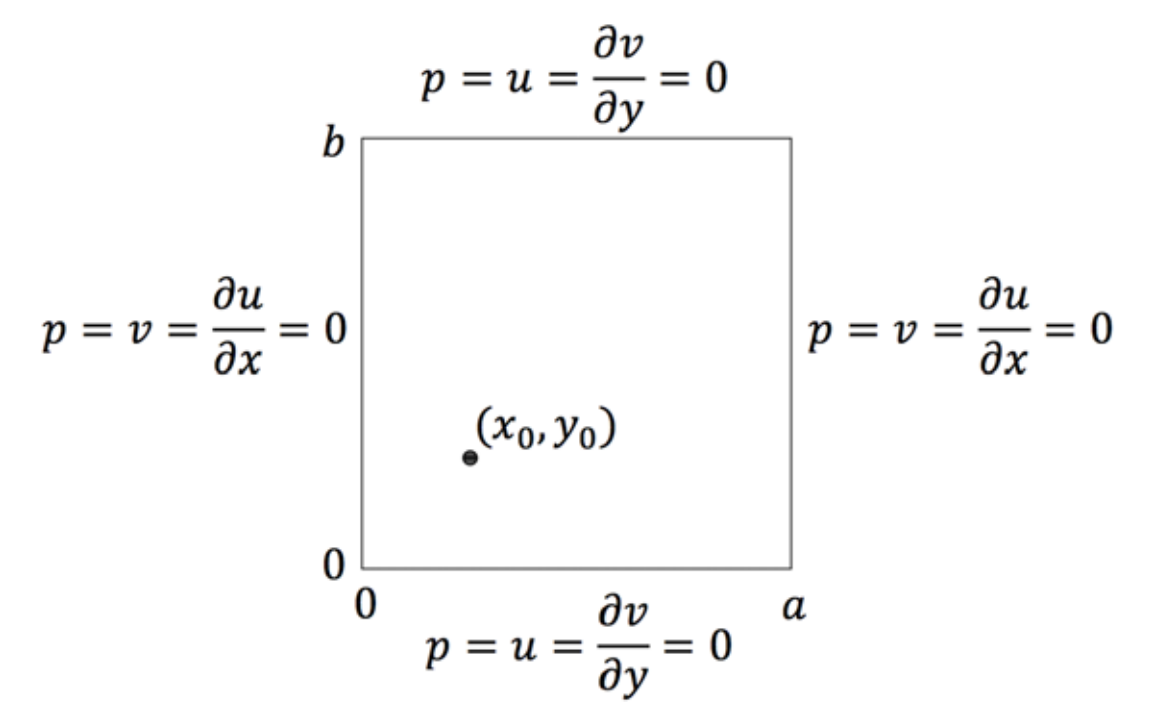}
\label{fig:2}
 \end{figure}

 As in \cite{Phillips07a,Rodrigo16}, 
 we consider the rectangular domain $(0,1)\times (0,1)$, and the following values of the material parameters:
 \[
c_0=0,\;\;\alpha = 1,\;\;  E = 10^5,\;\; \nu = 0.1,\;\; \kappa = 10^{-2},
 \]
where $E$ and $\nu$ denotes Young's modulus and the Poisson ratio, respectively, and 
\[
 \mu = \frac{E}{2(1+\nu)},\quad 
 \lambda = \frac{ E\nu}{(1-2\nu)(1+\nu)}.
\]
The source is positioned at the point $(1/4,1/4)$.

We consider the fully discrete scheme \eqref{bdf-time},  
with the (spatial) polynomial degree $k =1$ in the finite element spaces \eqref{space}
and the (temporal) BDF2 method ($m=2$). We use a relatively large time step of $\Delta t =\frac{\pi}{20 \beta}$.
The solution for the pressure on the deformed domain
on a uniform triangular mesh with mesh size $h = 1/64$ is plotted in Figure \ref{fig:3} 
for two different ``normalized time'' $\hat t = \beta t$ of values $\hat t = \pi/2$ and
$\hat t = 3\pi/2$. We observe that depending on the sign of the source term (positive for $\hat t = \pi/2$, negative for $\hat t=3\pi/2$)
the resultant displacements cause an expansion or a contraction of the medium. 
We also plot the pressure and x-component of the 
displacement profiles on three 
consecutive meshes, with mesh size $h = 1/32, 1/64,1/128$, 
along the diagonal line (0,0)--(1,1) of the domain, along with the exact solution in Figure \ref{fig:4}.
We observe form Figure \ref{fig:4} that the numerical solution resemble the exact solution very precisely.
\begin{figure}[ht!]
 \caption{Numerical solution for the pressure on the deformed domain at different time}
 \begin{multicols}{2}
  \includegraphics[width=0.48\textwidth]{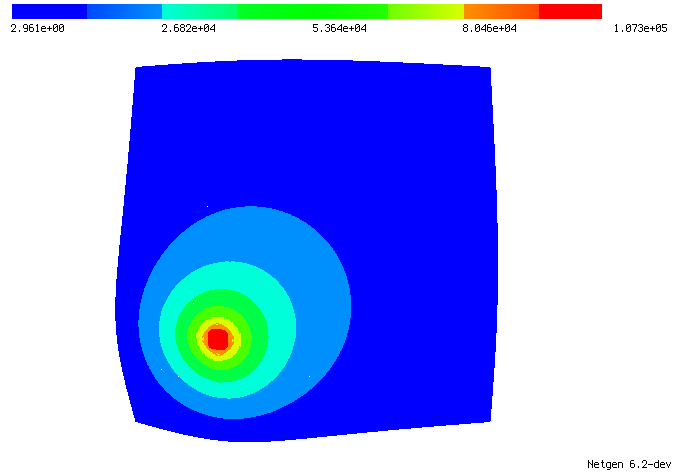}
  \begin{center}
   (a) $\hat t = \pi/2$
  \end{center}
 \columnbreak
 
  \includegraphics[width=0.48\textwidth]{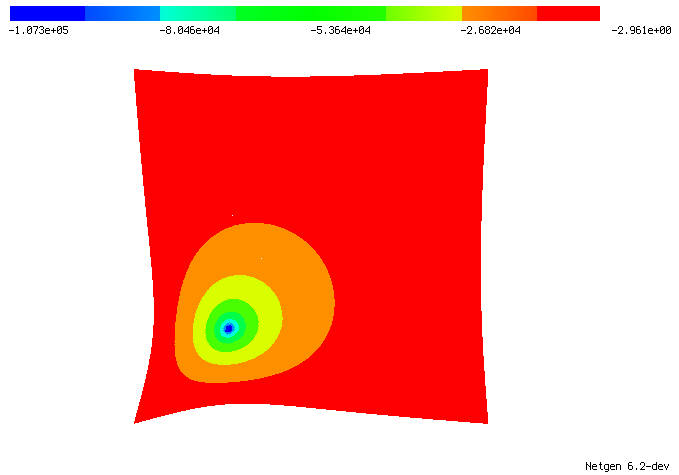} 
  \begin{center}
   (b) $\hat t = 3\pi/2$
  \end{center}
 \end{multicols}
\label{fig:3}
 \end{figure}

 \begin{figure}[ht!]
 \caption{Numerical solution for pressure and x-component of displacement along the diagonal (0,0)-(1,1) of the domain for different time}
 \begin{multicols}{2}
  \includegraphics[width=.5\textwidth]{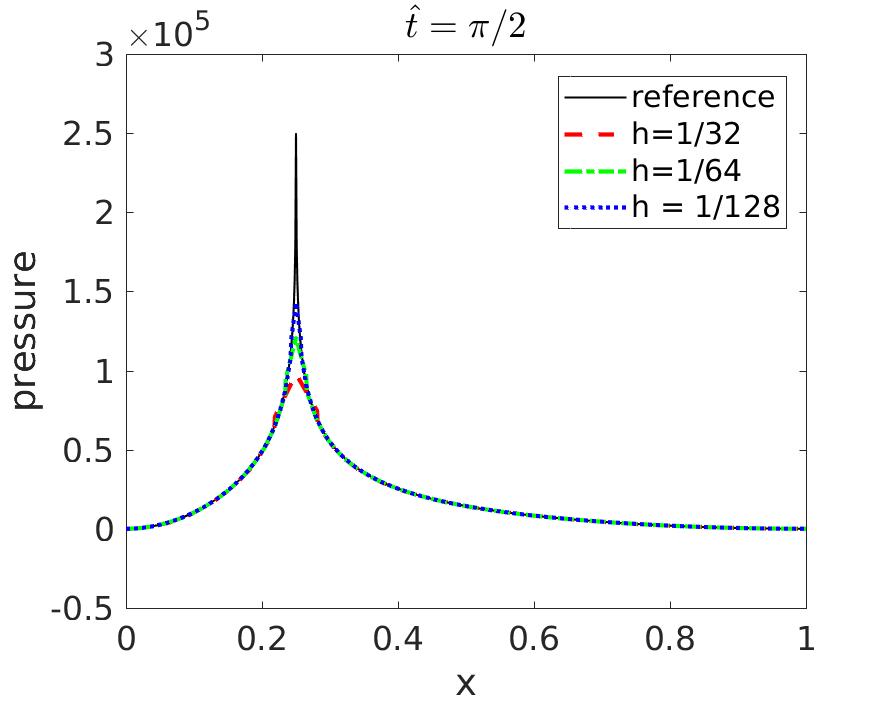}\\
  \includegraphics[width=.5\textwidth]{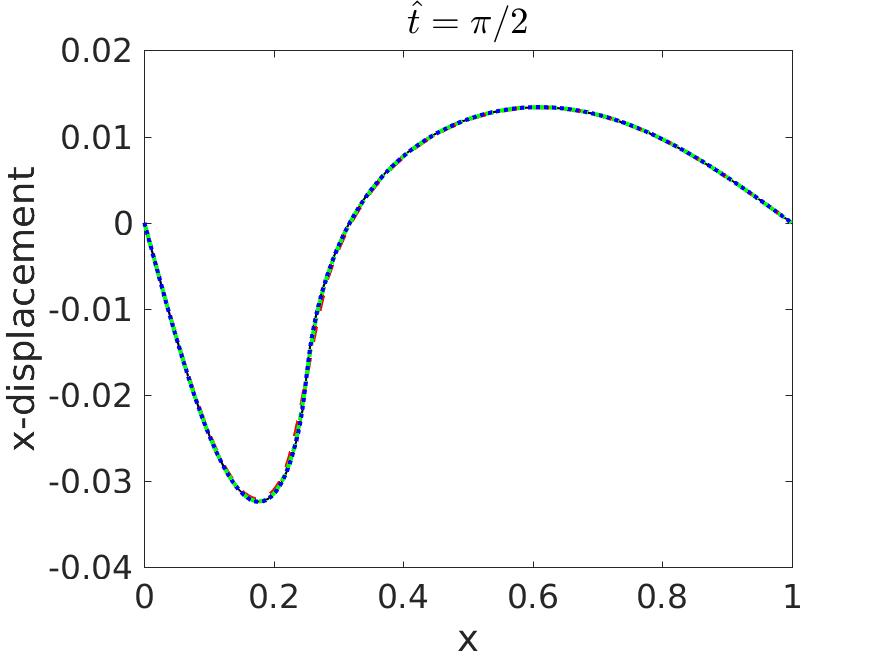}
 \columnbreak

   \includegraphics[width=.5\textwidth]{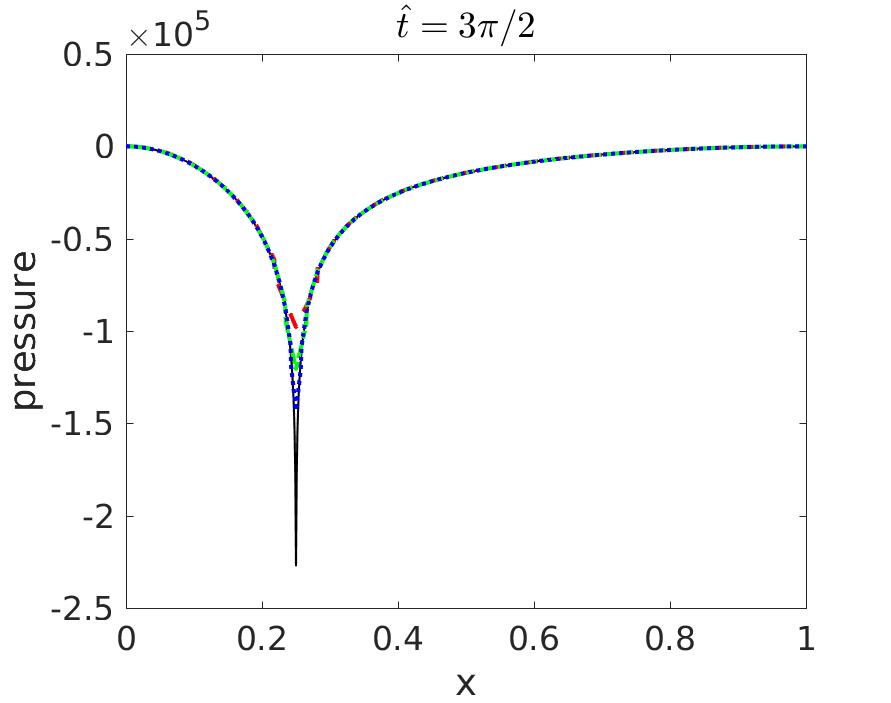}\\
   \includegraphics[width=.5\textwidth]{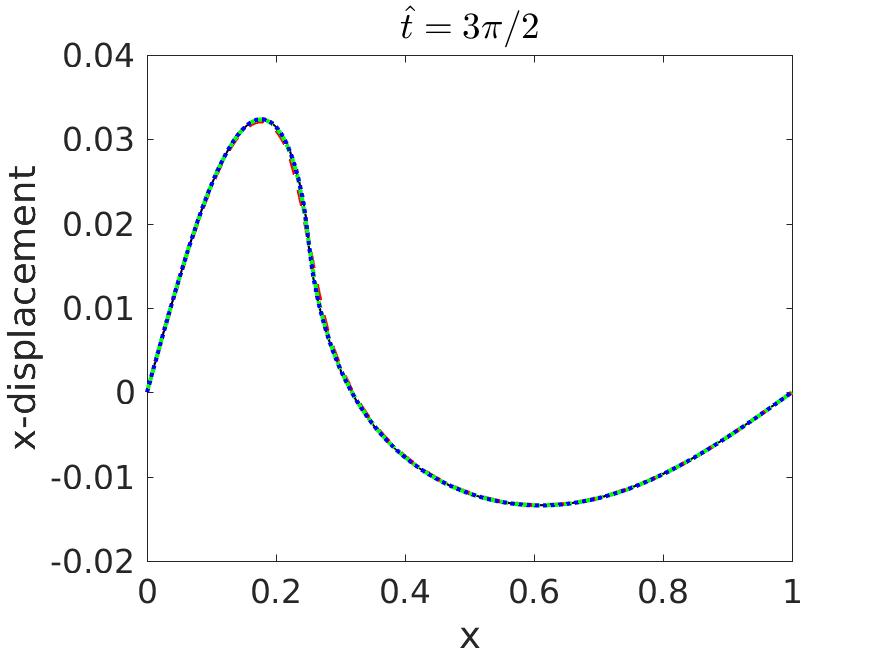}
 \end{multicols}
\label{fig:4}
 \end{figure}

 Finally, to check the robustness of the method with respect to pressure oscillations for small permeability combined with small time steps, 
 we show in Figure \ref{fig:5} the pressure profile after one step of backward Euler with $\kappa = 10^{-6}$ and $\Delta t = 10^{-4}$ on 
 the uniform triangular mesh with $h = 1/64$. We do not observe significant oscillation.

 \begin{figure}[ht!]
    \caption{Numerical solution for pressure after one time step. Left: numerical pressure on $\Omega$; Right: 
    numerical pressure on the diagonal line (0,0)--(1,1).}
     \includegraphics[width=.43\textwidth]{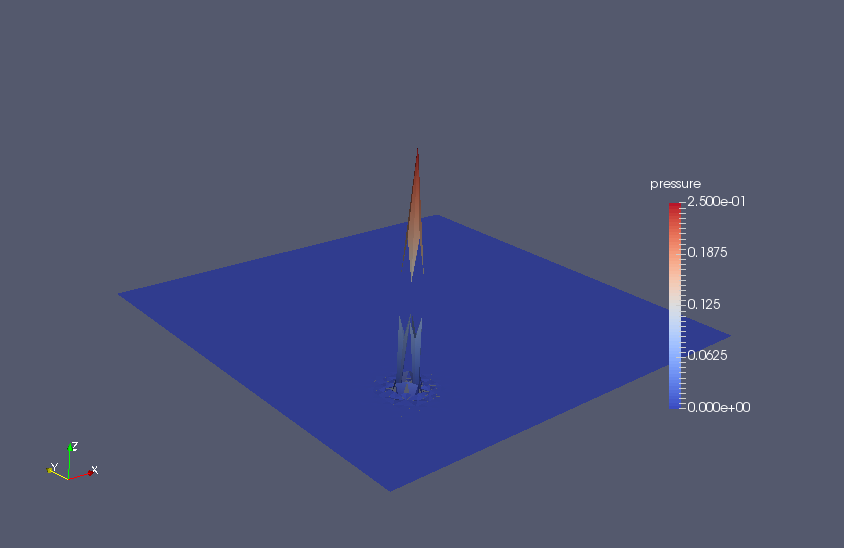}
   \includegraphics[width=.43\textwidth]{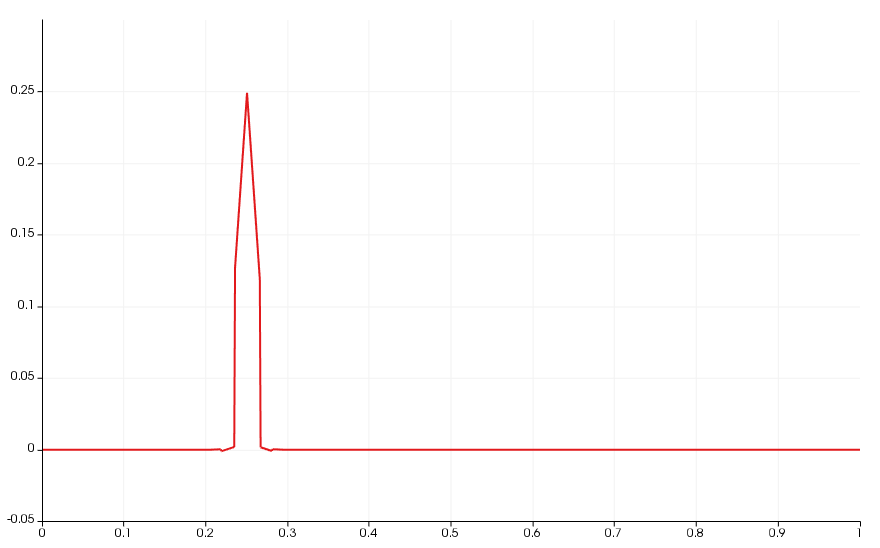}
   \label{fig:5}
 \end{figure}

 \section{Conclusion}
 \label{sec:conclusion}
 In this paper we have analyzed the convergence property of a novel high-order 
 HDG discretization of Biot's consolidation model in poroelasticity combined 
 with BDF time stepping.
 The method produce optimal convergence rates, and is free from Poisson locking when $\lambda \rightarrow \infty$.

\bibliographystyle{siam}


\begin{thebibliography}{10}

\bibitem{Akrivis15}
{\sc G.~Akrivis}, {\em Stability of implicit-explicit backward difference
  formulas for nonlinear parabolic equations}, SIAM J. Numer. Anal., 53 (2015),
  pp.~464--484.

\bibitem{Lee17}
{\sc T.~Baerland, J.~J. Lee, K.-A. Mardal, and R.~Winther}, {\em Weakly imposed
  symmetry and robust preconditioners for {B}iot's consolidation model},
  Comput. Methods Appl. Math., 17 (2017), pp.~377--396.

\bibitem{BarryMercer99}
{\sc S.~I. Barry and G.~N. Mercer}, {\em Exact solutions for two-dimensional
  time-dependent flow and deformation within a poroelastic medium}, Trans. ASME
  J. Appl. Mech., 66 (1999), pp.~536--540.

\bibitem{Berger15}
{\sc L.~Berger, R.~Bordas, D.~Kay, and S.~Tavener}, {\em Stabilized
  lowest-order finite element approximation for linear three-field
  poroelasticity}, SIAM J. Sci. Comput., 37 (2015), pp.~A2222--A2245.

\bibitem{Berger17}
\leavevmode\vrule height 2pt depth -1.6pt width 23pt, {\em A stabilized finite
  element method for finite-strain three-field poroelasticity}, Comput. Mech.,
  60 (2017), pp.~51--68.

\bibitem{Biot41}
{\sc M.~A. Biot}, {\em General theory of three-dimensional consolidation}, J.
  Appl. Phys., 12 (1941), pp.~155--164.

\bibitem{Biot56}
\leavevmode\vrule height 2pt depth -1.6pt width 23pt, {\em Theory of
  deformation of a porous viscoelastic anisotropic solid}, J. Appl. Phys., 27
  (1956), pp.~459--467.

\bibitem{Biot72}
\leavevmode\vrule height 2pt depth -1.6pt width 23pt, {\em Theory of finite
  deformations of pourous solids}, Indiana Univ. Math. J., 21 (1971/72),
  pp.~597--620.

\bibitem{Boffi16}
{\sc D.~Boffi, M.~Botti, and D.~A. Di~Pietro}, {\em A nonconforming high-order
  method for the {B}iot problem on general meshes}, SIAM J. Sci. Comput., 38
  (2016), pp.~A1508--A1537.

\bibitem{BoffiBrezziFortin13}
{\sc D.~Boffi, F.~Brezzi, and M.~Fortin}, {\em Mixed finite element methods and
  applications}, vol.~44 of Springer Series in Computational Mathematics,
  Springer, Heidelberg, 2013.

\bibitem{CarpenterGottliebAbarbanelDon95}
{\sc M.~H. Carpenter, D.~Gottlieb, S.~Abarbanel, and W.~S. Don}, {\em The
  theoretical accuracy of {R}unge-{K}utta time discretizations for the
  initial-boundary value problem: a study of the boundary error}, SIAM J. Sci.
  Comput., 16 (1995), pp.~1241--1252.

\bibitem{Chen13}
{\sc Y.~Chen, Y.~Luo, and M.~Feng}, {\em Analysis of a discontinuous {G}alerkin
  method for the {B}iot's consolidation problem}, Appl. Math. Comput., 219
  (2013), pp.~9043--9056.

\bibitem{DiPietroErn10}
{\sc D.~A. Di~Pietro and A.~Ern}, {\em Discrete functional analysis tools for
  discontinuous {G}alerkin methods with application to the incompressible
  {N}avier-{S}tokes equations}, Math. Comp., 79 (2010), pp.~1303--1330.

\bibitem{Ferronato10}
{\sc M.~Ferronato, N.~Castelletto, and G.~Gambolati}, {\em A fully coupled
  3-{D} mixed finite element model of {B}iot consolidation}, J. Comput. Phys.,
  229 (2010), pp.~4813 -- 4830.

\bibitem{FuLehrenfeld18}
{\sc G.~Fu and C.~Lehrenfeld}, {\em A strongly conservative hybrid {DG}/{M}ixed
  {FEM} for the coupling of {S}tokes and {D}arcy flow},  (2017).
\newblock Submitted.

\bibitem{FuLehrenfeld18a}
\leavevmode\vrule height 2pt depth -1.6pt width 23pt, {\em An locking-free
  divergence-conforming hdg method for linear elasticity},  (2018).
\newblock In preparation.

\bibitem{HairerWanner10}
{\sc E.~Hairer and G.~Wanner}, {\em Solving ordinary differential equations.
  {II}}, vol.~14 of Springer Series in Computational Mathematics,
  Springer-Verlag, Berlin, 2010.

\bibitem{HeywoodRannacher90}
{\sc J.~G. Heywood and R.~Rannacher}, {\em Finite-element approximation of the
  nonstationary {N}avier-{S}tokes problem. {IV}. {E}rror analysis for
  second-order time discretization}, SIAM J. Numer. Anal., 27 (1990),
  pp.~353--384.

\bibitem{Hu17}
{\sc X.~Hu, C.~Rodrigo, F.~J. Gaspar, and L.~T. Zikatanov}, {\em A
  nonconforming finite element method for the {B}iot's consolidation model in
  poroelasticity}, J. Comput. Appl. Math., 310 (2017), pp.~143--154.

\bibitem{JaustSchutz14}
{\sc A.~Jaust and J.~Sch\"utz}, {\em A temporally adaptive hybridized
  discontinuous {G}alerkin method for time-dependent compressible flows},
  Comput. \& Fluids, 98 (2014), pp.~177--185.

\bibitem{Korsawe05}
{\sc J.~Korsawe and G.~Starke}, {\em A least-squares mixed finite element
  method for {B}iot's consolidation problem in porous media}, SIAM J. Numer.
  Anal., 43 (2005), pp.~318--339.

\bibitem{Lederer17}
{\sc P.~L. Lederer, C.~Lehrenfeld, and J.~Sch{\"o}berl}, {\em Hybrid
  discontinuous galerkin methods with relaxed h(div)-conformity for
  incompressible flows. part i}, arXiv preprint arXiv:1707.02782,  (2017).

\bibitem{Lee16}
{\sc J.~J. Lee}, {\em Robust error analysis of coupled mixed methods for
  {B}iot's consolidation model}, J. Sci. Comput., 69 (2016), pp.~610--632.

\bibitem{Lee17b}
{\sc J.~J. Lee, K.-A. Mardal, and R.~Winther}, {\em Parameter-robust
  discretization and preconditioning of {B}iot's consolidation model}, SIAM J.
  Sci. Comput., 39 (2017), pp.~A1--A24.

\bibitem{Lehrenfeld:10}
{\sc C.~Lehrenfeld}, {\em Hybrid {Discontinuous} {Galerkin} methods for solving
  incompressible flow problems}, 2010.
\newblock Diploma Thesis, MathCCES/IGPM, RWTH Aachen.

\bibitem{LehrenfeldSchoberl16}
{\sc C.~Lehrenfeld and J.~Sch{\"o}berl}, {\em High order exactly
  divergence-free hybrid discontinuous galerkin methods for unsteady
  incompressible flows}, Computer Methods in Applied Mechanics and Engineering,
  307 (2016), pp.~339--361.

\bibitem{Lewis98}
{\sc R.~W. Lewis and B.~A. Schrefler}, {\em The finite element method in the
  static and dynamic deformation and consolidation of porous media}, vol.~37 of
  Wiley Ser. Number. Methods Engrg., John Wiley, New York, 1998.

\bibitem{Liu09}
{\sc R.~Liu, M.~F. Wheeler, C.~N. Dawson, and R.~H. Dean}, {\em On a coupled
  discontinous/continuous {G}alerkin framework and an adaptive penalty scheme
  for poroelasticity problems}, Comput. Methods Appl. Mech. Engrg., 198 (2009),
  pp.~3499--3510.

\bibitem{LubichMansour13}
{\sc C.~Lubich, D.~Mansour, and C.~Venkataraman}, {\em Backward difference time
  discretization of parabolic differential equations on evolving surfaces}, IMA
  J. Numer. Anal., 33 (2013), pp.~1365--1385.

\bibitem{MuradLoula92}
{\sc M.~A. Murad and A.~F.~D. Loula}, {\em Improved accuracy in finite element
  analysis of {B}iot's consolidation problem}, Comput. Methods Appl. Mech.
  Engrg., 95 (1992), pp.~359--382.

\bibitem{MuradLoula94}
\leavevmode\vrule height 2pt depth -1.6pt width 23pt, {\em On stability and
  convergence of finite element approximations of {B}iot's consolidation
  problem}, Internat. J. Numer. Methods Engrg., 37 (1994), pp.~645--667.

\bibitem{Murad96}
{\sc M.~A. Murad, V.~Thom\'ee, and A.~F.~D. Loula}, {\em Asymptotic behavior of
  semidiscrete finite-element approximations of {B}iot's consolidation
  problem}, SIAM J. Numer. Anal., 33 (1996), pp.~1065--1083.

\bibitem{NevanlinnaOdeh81}
{\sc O.~Nevanlinna and F.~Odeh}, {\em Multiplier techniques for linear
  multistep methods}, Numer. Funct. Anal. Optim., 3 (1981), pp.~377--423.

\bibitem{NguyenPeraireCockburn11b}
{\sc N.~C. Nguyen, J.~Peraire, and B.~Cockburn}, {\em High-order implicit
  hybridizable discontinuous {G}alerkin methods for acoustics and
  elastodynamics}, J. Comput. Phys., 230 (2011), pp.~3695--3718.

\bibitem{Oikawa:15}
{\sc I.~Oikawa}, {\em A hybridized discontinuous {G}alerkin method with reduced
  stabilization}, J. Sci. Comput., 65 (2015), pp.~327--340.

\bibitem{Oyarzua16}
{\sc R.~Oyarz\'ua and R.~Ruiz-Baier}, {\em Locking-free finite element methods
  for poroelasticity}, SIAM J. Numer. Anal., 54 (2016), pp.~2951--2973.

\bibitem{Phillips05}
{\sc P.~J. Phillips}, {\em Finite element methods in linear poroelasticity:
  {T}heoretical and computational results}, ProQuest LLC, Ann Arbor, MI, 2005.
\newblock Thesis (Ph.D.)--The University of Texas at Austin.

\bibitem{Phillips07a}
{\sc P.~J. Phillips and M.~F. Wheeler}, {\em A coupling of mixed and continuous
  {G}alerkin finite element methods for poroelasticity. {I}. {T}he continuous
  in time case}, Comput. Geosci., 11 (2007), pp.~131--144.

\bibitem{Phillips07b}
{\sc P.~J. Phillips and M.~F. Wheeler}, {\em A coupling of mixed and continuous
  {G}alerkin finite element methods for poroelasticity. {II}. {T}he
  discrete-in-time case}, Comput. Geosci., 11 (2007), pp.~145--158.

\bibitem{Phillips08}
{\sc P.~J. Phillips and M.~F. Wheeler}, {\em A coupling of mixed and
  discontinuous {G}alerkin finite-element methods for poroelasticity}, Comput.
  Geosci., 12 (2008), pp.~417--435.

\bibitem{Phillips09}
\leavevmode\vrule height 2pt depth -1.6pt width 23pt, {\em Overcoming the
  problem of locking in linear elasticity and poroleasticity: an heuristic
  approach}, Comput. Geosci., 13 (2009), pp.~5--12.

\bibitem{Riviere17}
{\sc B.~Rivi\`ere, J.~Tan, and T.~Thompson}, {\em Error analysis of primal
  discontinuous {G}alerkin methods for a mixed formulation of the {B}iot
  equations}, Comput. Math. Appl., 73 (2017), pp.~666--683.

\bibitem{Rodrigo16}
{\sc C.~Rodrigo, F.~J. Gaspar, X.~Hu, and L.~T. Zikatanov}, {\em Stability and
  monotonicity for some discretizations of the {B}iot's consolidation model},
  Comput. Methods Appl. Mech. Engrg., 298 (2016), pp.~183--204.

\bibitem{RosalesSeiboldShirokoffZhou17}
{\sc R.~R. Rosales, B.~Seibold, D.~Shirokoff, and D.~Zhou}, {\em Order
  reduction in high-order {Runge-Kutta} methods for initial boundary value
  problems}, arXiv:1712.00897 [math.NA].
\newblock submitted on Dec. 4 2017.

\bibitem{Schoberl16}
{\sc J.~Sch{\"o}berl}, {\em {C}++11 {I}mplementation of {F}inite {E}lements in
  {NGS}olve}, 2014.
\newblock {ASC Report 30/2014, Institute for Analysis and Scientific Computing,
  Vienna University of Technology}.

\bibitem{Strehlow15}
{\sc K.~Strehlow, J.-H. Gottsmann, and A.-C. Rust}, {\em Poroelastic responses
  of confined aquifers to subsurface strain and their use for volcano
  monitoring}, Solid Earth., 6 (2015), pp.~1207--1229.

\bibitem{Wan02}
{\sc J.~Wan}, {\em Stabilized finite element method for coupled geomechanics
  and multiphase flow}, 2002.
\newblock Ph.D. Thesis, Stanford University, Stanford, CA.

\bibitem{Wheeler14a}
{\sc M.~Wheeler, G.~Xue, and I.~Yotov}, {\em Coupling multipoint flux mixed
  finite element methods with continuous {G}alerkin methods for
  poroelasticity}, Comput. Geosci., 18 (2014), pp.~57--75.

\bibitem{Wheeler75}
{\sc M.~F. Wheeler}, {\em An {$H^{-1}$} {G}alerkin method for parabolic
  problems in a single space variable}, SIAM J. Numer. Anal., 12 (1975),
  pp.~803--817.

\bibitem{White08}
{\sc J.~A. White and R.~I. Borja}, {\em Stabilized low-order finite elements
  for coupled solid-deformation/fluid-diffusion and their application to fault
  zone transients}, Comput. Methods Appl. Mech. Engrg., 197 (2008),
  pp.~4353--4366.

\bibitem{Yi13}
{\sc S.-Y. Yi}, {\em A coupling of nonconforming and mixed finite element
  methods for {B}iot's consolidation model}, Numer. Methods Partial
  Differential Equations, 29 (2013), pp.~1749--1777.

\bibitem{Yi14}
\leavevmode\vrule height 2pt depth -1.6pt width 23pt, {\em Convergence analysis
  of a new mixed finite element method for {B}iot's consolidation model},
  Numer. Methods Partial Differential Equations, 30 (2014), pp.~1189--1210.

\bibitem{Y117}
\leavevmode\vrule height 2pt depth -1.6pt width 23pt, {\em A study of two modes
  of locking in poroelasticity}, SIAM J. Numer. Anal., 55 (2017),
  pp.~1915--1936.

\bibitem{Young14}
{\sc J.~Young, B.~Rivi\`ere, C.~S. Cox, Jr., and K.~Uray}, {\em A mathematical
  model of intestinal oedema formation}, Math. Med. Biol., 31 (2014),
  pp.~1--15.

\bibitem{Zheng03}
{\sc Y.~Zheng, R.~Burridge, and D.~Rurns}, {\em Reservoir simulation with the
  finite element method using biot poroelastic approach}, Massachusetts
  Institute of Technology. Earth Resources Laboratory, 2003.

\end{thebibliography}

\end{document}